\documentclass[12pt]{article}
\usepackage[amsmath]{e-jc}

\usepackage{graphicx}
\usepackage{tikz,color,hyperref}
\usepackage{subfigure}
\usepackage{graphicx}
\usepackage{mathtools}
\usepackage{dsfont}
\usepackage{enumitem}
\usetikzlibrary{intersections,patterns,arrows,decorations.pathreplacing,shapes.misc}

 \newtheorem{algorithm}[theorem]{Algorithm}

\newcommand{\beq}{\begin{equation}}
\newcommand{\eeq}{\end{equation}}

\newcommand{\Config}[1]{\mathrm{Config}\left( #1 \right)}
\newcommand{\Stable}[1]{\mathrm{Stable}\left( #1 \right)}
\newcommand{\StoRec}[1]{\mathrm{StoRec}\left( #1 \right)}
\newcommand{\DetRec}[1]{\mathrm{DetRec}\left( #1 \right)}
\newcommand{\DRPoly}[1]{\mathrm{DRPoly}_{#1}}
\newcommand{\level}{\mathrm{level}}

\newcommand{\Confign}{\mathrm{Config}_n}
\newcommand{\Stablen}{\mathrm{Stable}_n}
\newcommand{\StoRecn}{\mathrm{StoRec}_n}
\newcommand{\DetRecn}{\mathrm{DetRec}_n}
\newcommand{\PSR}[1]{\mathrm{PartStoRec}_n^{(#1)}}

\newcommand{\OO}{\mathcal{O}}
\newcommand{\In}[1]{\mathrm{in}^{\OO}_{#1}}
\newcommand{\Out}[1]{\mathrm{out}^{\OO}_{#1}}
\newcommand{\orighta}{\xrightarrow{\OO}}
\newcommand{\olefta}{\xleftarrow{\OO}}

\newcommand{\dgr}[1]{\mathrm{deg}_{#1}}

\newcommand{\inc}[1]{{#1}^{\mathrm{inc}}}

\newcommand{\DT}{\mathrm{DetTopp}}
\newcommand{\DetStab}{\mathrm{DetStab}}

\newcommand{\ST}{\mathrm{StoTopp}}
\newcommand{\StoStab}{\mathrm{StoStab}}

\newcommand{\PF}[1]{\mathrm{PF}_{#1}}
\newcommand{\PFPoly}[1]{\mathrm{PFPoly}_{#1}}
\newcommand{\PermPoly}[1]{\mathrm{PermPoly}_{#1}}

\newcommand{\maxi}{\mathrm{max}}
\newcommand{\nextMax}{\mathrm{nextMax}}
\newcommand{\maxc}{c^{\max}}

\newcommand{\R}{\mathbb{R}}
\newcommand{\Z}{\mathbb{Z}}
\newcommand{\Zp}{\mathbb{Z}_+}
\newcommand{\N}{\mathbb{N}}

\dateline{Feb 7, 2024}{Aug 21, 2024}{Sep 6, 2024}


\Copyright{The author. Released under the CC BY license (International 4.0).}

\title{The stochastic sandpile model on complete graphs}

\author{Thomas Selig
}

\authortext{}{Department of Computing, School of Advanced Technology, Xi'an Jiaotong-Liverpool University, Suzhou, China
   (\email{Thomas.Selig@xjtlu.edu.cn}).}
   
\begin{document}

\maketitle

\begin{abstract}

The stochastic sandpile model (SSM) is a generalisation of the standard Abelian sandpile model (ASM), in which topplings of unstable vertices are made random. When unstable, a vertex sends one grain to each of its neighbours independently with probability $p \in (0,1)$. We study the SSM on complete graphs. Our main result is a description of the recurrent states of the model. We show that these are given by convex sums of recurrent states for the ASM. This allows us to recover a well-known result: that the number of integer lattice points in the $n$-dimensional permutation polytope is equal to the number of labeled spanning forests on $n$ vertices. We also provide a stochastic version of Dhar's burning algorithm to check if a given (stable) state is recurrent or not, which runs in linear time. Finally, we study a family of so-called \emph{partial} SSMs, in which some vertices topple randomly, while others topple deterministically (as in the ASM, sending one grain to all neighbours). We show that this distinction is meaningful, yielding sets of recurrent states that are in general different from those of both the ASM and SSM. We also show that to get all recurrent states of the SSM, we can allow up to two vertices to topple deterministically.
\end{abstract}


\section{Introduction}\label{sec:intro}

The Abelian sandpile model (ASM) was introduced by Bak, Tang and Wiesenfeld in the late 1980's~\cite{BTW1, BTW2} as an example of a model exhibiting a phenomenon known as \emph{self-organised criticality}. This phenonemon describes systems which self-tune themselves towards some critical state, without the need for external modification of their parameters. The model was then generalised and formalised a few years later by Dhar~\cite{Dhar1}.

The ASM is a random process on a graph, where vertices are assigned a number of grains of sand (a non-negative integer). At each unit of time, a grain is added to a randomly selected vertex of the graph. If this causes the number of grains to exceed a certain threshold (usually the degree of the vertex), that vertex is said to be \emph{unstable}, and \emph{topples}, sending one grain to each of its neighbours in the graph. A special vertex called the \emph{sink} absorbs grains (never topples), and so the process eventually stabilises.

Of particular interest in this model are the so-called \emph{recurrent states}, which are states that appear infinitely often in the long-time running of the model. While there are a number of ways to check if a given state is recurrent or not (Theorem~\ref{thm:DR_general} gives four of these), computing the set of all recurrent states is in general a difficult question, both algorithmically and combinatorially. As such, a fruitful direction of ASM research has been to instead focus on certain graph families with high levels of symmetry or structure, on which the set of recurrent states can be more easily studied and computed.

The seminal example of such a study is due to Cori and Rossin~\cite{CR}, who showed that on complete graphs, the set of recurrent states is in bijection with the set of \emph{parking functions} (see Section~\ref{subsec:parking_functions} for a definition of these, and Theorem~\ref{thm:detrec_pf} for a statement of the bijection). Similar combinatorial studies on many other graph families -- such as complete bipartite~\cite{DLB} and multi-partite~\cite{CorPou} graphs, complete split graphs~\cite{Duk}, wheel and fan graphs~\cite{SelWheel}, Ferrers graphs~\cite{DSSS1}, permutation graphs~\cite{DSSS2}, and so on -- followed. Related combinatorial objects include generalisations of parking functions, parallelogram polyominoes, Motzkin words, subgraphs of cycles, lattice paths, decorated EW-tableaux, tiered trees, and a number of others.

In the ASM, the only randomness lies in the choice of vertex where grains are added at each time step, together with possibly the randomness of the initial state. After this, the toppling and stabilisations processes are entirely deterministic. Stochastic variants of the ASM add an extra layer of randomness to the model by making topplings random. When a vertex is unstable, it chooses a random (multi-)subset of its neighbours to send grains to according to some probability distribution. There are a number of stochastic variants of the ASM in the literature, including the following.
\begin{itemize}
\item In~\cite{Manna}, there are two different types of grain, which cannot occupy the same vertex. When they do, one of the grains is moved to a randomly chosen neighbouring vertex instead.
\item In~\cite{DharSad}, unstable vertices lose a fixed number grains, which are re-distributed at random to their neighbours: some neighbours may receive more than one grain, others none, while some grains may exit the system directly.
\item In~\cite{CMS}, unstable vertices flip a biased coin for each neighbour to decide which neighbours to send grains to. That is, all neighbours independently of each other receive a grain with probability $p  \in (0,1)$ (with probability $(1-p)$ that grain is kept by the toppling vertex).
\item The model in~\cite{Nunzi} generalises the two previous models in~\cite{CMS, DharSad}. In this model, toppling vertices send grains to a random multi-set of neighbours, with the total number of grains lost being itself random.
\item In~\cite{KW}, the toppling threshold of a vertex is set to a (fixed) multiple $M$ of its degree. For each toppling, a random number $\gamma \in \{1,\ldots,M\}$ is chosen, and each neighbour of the toppling vertex receives the same (random) number $\gamma$ of grains.
\end{itemize}

In general, these models have not been studied as widely as the ASM. The existing research has mainly focused on the physical properties of the models such as the steady state distribution, and no real combinatorial studies have been conducted. This paper proposes the first such combinatorial study. We focus on the ``coin-flipping'' model from~\cite{CMS}, which we call the stochastic sandpile model (SSM). We choose this model because it has known characterisations of recurrent states, i.e.\ conditions under which a given state is recurrent (see Theorem~\ref{thm:SR_general} in this paper). We study the SSM on complete graphs, seeking a combinatorial description of its set of recurrent states in the spirit of that of Cori and Rossin for the ASM~\cite{CR}.

Our paper is organised as follows. In Section~\ref{sec:prelim} we introduce the necessary definitions, tools and notation for our study. This includes formal definitions of the ASM and SSM, and characterisations of their recurrent states (Theorems~\ref{thm:DR_general} and \ref{thm:SR_general}). We also introduce various notions such as orientations, parking functions, and polytopes, which will be needed later in the paper. 
In Section~\ref{sec:sandpile_complete} we focus on the sandpile model (both Abelian and stochastic) on complete graphs. We re-state the characterisation theorems mentioned above in terms of complete graphs, and this allows us to define a stochastic burning algorithm that establishes in log-linear time whether a given state is recurrent for the SSM or not (Theorem~\ref{thm:sto_burning}). 
In Section~\ref{sec:main_result} we state and prove our main combinatorial description of recurrent states for the SSM on complete graphs (Theorem~\ref{thm:SR_DRPoly}), namely that these are given by integer lattice points in the convex polytope of the set of recurrent states for the ASM.
In Section~\ref{sec:partial_ssm} we introduce a concept of \emph{partially} stochastic sandpile models in which some vertices topple according to the SSM rules, while others topple according to the ASM rules. We show that the recurrent states for these models are in general distinct to those of both the ASM and SSM (Propositions~\ref{pro:part_sto_non_det} and \ref{pro:part_sto_non_sr}), and that if at most two vertices topple according to the ASM rules, we in fact get all the recurrent states of the SSM (Theorem~\ref{thm:par_sto_rec_all_sto_rec}).
Finally, Section~\ref{sec:conc} summarises our main results, and lists some potential directions of future research.


\section{Preliminaries}\label{sec:prelim}

In this section, we introduce some of the necessary definitions, tools and notation we will need and use throughout the rest of the paper. As usual, the sets $\Z$ and $\N$ denote the sets of integers and (strictly) positive integers respectively. We let $\Zp:=\N \, \cup \, \{0\}$ denote the set of non-negative integers. For a positive integer $n \in \N$, we denote $[n]$, resp.\ $[n]_0$, the set $\{1,\ldots,n\}$, resp.\ $\{0,\ldots,n\}$. For a vector $v = (v_1, \ldots, v_n) \in \R^n$, we write $\inc{v} = (\inc{v}_1,\ldots,\inc{v}_n)$ for the non-decreasing rearrangement of $v$.

Throughout this section, a \emph{graph} $G$ is a labelled, undirected, connected graph with vertex set $[n]_0$. The edge set $E$ is finite, and may contain multiple edges, but no loops. We call vertex $0$ the \emph{sink} of the graph $G$. For $i \in [n]_0$, we write $\dgr{i} = \dgr{i}^G$ for the degree of the vertex $i$ in $G$, omitting the superscript where the underlying graph is unambiguous. For a subset of vertices $A \subseteq [n]_0$, we denote $G[A]$ the induced subgraph of $G$ on $A$, that is the graph with vertex set $A$ and edge set the set of edges of $G$ with both endpoints in $A$. The complete graph $K_n^0$ is the graph where there is one edge between any two distinct vertices $i,j \in [n]_0$.

\subsection{Orientations}\label{subsec:orient}

An \emph{orientation} $\OO$ of a graph $G$ is the assignment of a direction to each edge in $E$. If $E$ contains a multiple edge, a direction is assigned to each of its copies. For an orientation $\OO$, and an edge $e=\{i,j\} \in E$, we write $i \orighta j$ to denote that the edge $e$ is directed from $i$ to $j$ in $\OO$. For $i \in [n]_0$, we denote $\In{i}$, resp.\ $\Out{i}$, the number of incoming, resp.\ outgoing, edges at $i$ in $\OO$.

An orientation is \emph{acyclic} if it contains no directed cycles. A \emph{root}, resp.\ \emph{source}, of an orientation is a vertex where all edges are incoming, resp.\ outgoing. It is straightforward to verify that an acyclic orientation contains at least one root and one source. An orientation is said to be \emph{sink-rooted} if it contains exactly one root, the sink $0$. A \emph{cycle-flip} refers to the act of changing the directions of all edges in a directed cycle, leaving other edge directions unchanged.

\begin{definition}\label{def:equiv_orient}
We say that two orientations $\OO$ and $\OO'$ of $G$ are \emph{score-equivalent} if
\beq\label{eq:equiv_orient_def}
\forall\, i \in [n]_0, \, \In{i} = \mathrm{in}^{\OO'}_i.
\eeq
\end{definition}

The following is stated in equivalent form in \cite[Lemma 1]{Felsner}.

\begin{proposition}\label{pro:equiv_orient_cycle_flip}
Two orientations $\OO$ and $\OO'$ are score-equivalent if, and only if, the orientation $\OO'$ can be obtained from $\OO$ through a series of cycle-flips.
\end{proposition}

Note that as a consequence, if $\OO$ is acyclic, its score-equivalence class contains only itself.

\subsection{The Abelian sandpile model}\label{subsec:ASM_prelim}

In this part we introduce the Abelian sandpile model (ASM) on a graph, and recall some important results regarding the so-called \emph{recurrent states} of the model. Let $G$ be a graph with vertex set $[n]_0$ where $0$ is the sink, and edge set $E$.

A \emph{configuration} on $G$ is a vector $c=(c_1,\ldots ,c_n) \in \Zp^n$ that assigns the number $c_i$ to vertex $i$. 
We think of $c_i$ as representing the number of grains of sand at the vertex $i$. 
Note that the sink vertex $0$ is not assigned a number of grains.
Denote by $\Config{G}$ the set of all configurations on $G$.
Let $\alpha_i \in \Zp^n$ be the vector with $1$ in the $i$-th position and $0$ elsewhere. By convention, $\alpha_0 = (0, \ldots, 0)$ is the all-0 vector.

We say that a vertex $i$ in a configuration $c=(c_1,\ldots ,c_n)\in \Config{G}$ is \emph{stable} if $c_i < \dgr{i}^G$, and \emph{unstable} otherwise. 
A configuration is called stable if all its non-sink vertices are stable (otherwise it is unstable), and we denote $\Stable{G}$ the set of all stable configurations on $G$. We also define $\maxc$ to be the \emph{maximal stable configuration} on $G$, i.e.\ $\maxc_i = \dgr{i}^G - 1$ for all $i \in [n]$. The terminology ``maximal'' here means that adding a grain to any vertex in $\maxc$ would result in an unstable configuration.

Unstable vertices may topple. 
We define the \emph{deterministic toppling operator} $\DT_i$, corresponding to the toppling of an unstable vertex $i\in [n]$ in a configuration $c \in \Config{G}$, by:
\beq\label{eq:full_toppling}
\DT_i(c) := c - \dgr{i} \cdot \alpha_i + \sum_{j: \{i,j\} \in E} \alpha_j,
\eeq
where the sum is over all vertices $j$ adjacent to $i$ in $G$, counted with multiplicity. In words, the deterministic toppling of a vertex $i$ sends one grain along each edge incident to $i$ (with multiplicity), with the grains then being received by the neighbouring vertices.

Performing this toppling may cause other vertices to become unstable, and we topple these in turn. 
One can show (see e.g.\ \cite{Dhar}) that starting from some unstable configuration $c$ and successively toppling unstable vertices, we eventually reach a stable configuration $c'$ (we think of the sink as absorbing grains). 
Moreover, this configuration $c'$ does not depend on the sequence in which vertices are toppled. 
We write $c' = \DetStab(c)$ and call it the \emph{deterministic stabilisation} of $c$.

\begin{remark}\label{rem:det_sto_diff}
The qualifier \emph{deterministic} for toppling operators and stabilisation is used here to distinguish the ASM from its stochastic variant introduced in Section~\ref{subsec:SSM_prelim}. For this variant, we will talk instead of \emph{stochastic} topplings and stabilisation.
\end{remark}

We now define a Markov chain on the set of stable configurations $\Stable{G}$.
Let $\mu=(\mu_1,\dots,\mu_n)$ be a probability distribution on $[n]$ such that $\mu_i>0$ for all $i \in [n]$. 
At each step of the Markov chain we add a grain at the vertex $i$ with probability $\mu_i$ and (deterministically) stabilise the resulting configuration.
Formally the transition matrix $Q$ is given by:
\beq\label{eq:transition_matrix}
\forall\, c,c' \in \Stable{G}\!, \,  Q(c,c')=\sum\limits_{i=1}^n \mu_i \mathds{1}_{\DetStab \left(c+\alpha_i \right) \, = \, c'}.
\eeq

The recurrent states for the Markov chain are the set of configurations which appear infinitely often in the long-time running of the model. In the spirit of Remark~\ref{rem:det_sto_diff}, we call the recurrent states of the ASM \emph{deterministically recurrent} (DR), and let $\DetRec{G}$ be the set of DR states on the graph $G$. Since $\mu_i >0$ for all $i \in [n]$, it is clear that the maximal stable configuration $\maxc$ is DR, and that the Markov chain is irreducible.

Given a configuration $c = (c_1,\ldots,c_n) \in \Config{G}$ and an orientation $\OO$ of $G$, we say that $\OO$ and $c$ are \emph{compatible} if
\beq\label{eq:comp_orient_config}
\forall\, i \in [n],\, c_i \geq \In{i}.
\eeq
Note that if $\OO$ and $c$ are compatible, and $\OO'$ is an orientation of $G$ which is score-equivalent to $\OO$, then $\OO'$ and $c$ are also compatible. If $\OO$ and $c$ are compatible, we will sometimes say that $\OO$ is compatible with $c$, or simply that $\OO$ is compatible if there is no ambiguity over which configuration $c$ is considered.

The study of DR states is of central importance in ASM research. In Theorem~\ref{thm:DR_general} we give four equivalent characterisations of DR states.
\begin{enumerate}
\item The first is a simple Markov chain result, stating that recurrent states can be reached from the maximal configuration through a series of grain additions and deterministic stabilisations. The formulation used in the theorem stems from the remark that in this sequence, we can put all the grain additions first, and then it remains to effect just one stabilisation.
\item The second is in terms of compatible \emph{acyclic orientations}. This characterisation was first stated in these terms by Biggs \cite{Biggs}, although the author credits a previous paper \cite{GZ} as having equivalent results.
\item The third is in terms of \emph{forbidden subconfigurations}, essentially configurations on a strict subgraph of $G$ which remain stable (see \cite{Red}).
\item The fourth is the famous \emph{burning algorithm}, due to Dhar~\cite{Dhar}, which provides a straightforward algorithmic process to check if a given (stable) configuration is DR or not.
\end{enumerate}

That Characterisations (3) and (4) are equivalent is fairly straightforward, but we choose to give both here to foreshadow our results on the SSM. First, let us describe Dhar's burning algorithm. This is a process which burns vertices of $G$ until either no vertices are left, or a forbidden subconfiguration is reached. \emph{Burning} a vertex of a graph means removing that vertex and all its incident edges. Dhar's algorithm can be described as follows.

\begin{algorithm}[Dhar's burning algorithm]\label{algo:Dhar_burning}
Input: a stable configuration $c \in \Stable{G}$.
\begin{itemize}
\item \textbf{Step 1}. Burn the sink vertex $0$. Let $\mathrm{Remain} = [n]_0 \setminus \{0\} = [n]$ be the set of remaining (unburned) vertices, and $G[\mathrm{Remain}]$ be the induced subgraph of $G$ on the set of remaining vertices.
\item \textbf{Step 2}. While there exists $i \in \mathrm{Remain}$ such that $c_i \geq \dgr{i}^{G[\mathrm{Remain}]}$, burn vertex $i$, setting $\mathrm{Remain} = \mathrm{Remain} \setminus \{i\}$, and repeat.
\item \textbf{Step 3}. Output the set of remaining vertices $\mathrm{Remain} = \mathrm{Remain}(c)$.
\end{itemize}
\end{algorithm}

Traditionally, in Step~2, one chooses to burn the minimal $i$ satisfying the inequality, but the output does not depend on this choice in any way, so we do not specify it here. We can now state the four equivalent characterisations of DR states referred to above.

\begin{theorem}\label{thm:DR_general}
Let $c = (c_1,\ldots,c_n) \in \Stable{G}$ be a \emph{stable} configuration on the graph $G$. Then $c$ is DR if, and only if, one of the following four equivalent conditions holds.
\begin{enumerate}[label=(\roman*), itemsep=1pt, topsep = 5pt]
\item There exists a configuration $d \in \Config{G}$ such that $\DetStab(\maxc + d) = c$.
\item There exists an \emph{acyclic, sink-rooted} orientation $\OO$ of $G$ compatible with $c$.
\item For every subset $A \subseteq [n]$, there exists a vertex $i \in A$ such that $c_i \geq \dgr{i}^{G[A]}$.
\item Dhar's burning algorithm burns all vertices of $G$, i.e. outputs $\mathrm{Remain}(c) = \emptyset$.
\end{enumerate}
\end{theorem}

\subsection{The stochastic sandpile model}\label{subsec:SSM_prelim}

In this section, we introduce a stochastic variant of the ASM, called stochastic sandpile model (SSM). We are slightly more informal than in the previous section, as most of the notions introduced closely mirror those of the ASM. For a more formal definition of the SSM, see~\cite{CMS}. 

In the ASM, the only randomness in the model concerns the choice of which vertex to add a grain to at each step of the Markov chain (i.e.\ the distribution $\mu$ in Section~\ref{subsec:ASM_prelim}), while the subsequent operations -- topplings and thus stabilisation -- are deterministic.

In the SSM, we introduce an additional layer of randomness by making the topplings random. Fix some parameter $p \in (0,1)$. Informally, every time we have an unstable vertex $i$, we flip a biased coin for each incident edge to $i$, and decide to send one grain along that edge with probability $p$ (the grain gets sent to one of the neighbours of $i$), while with probability $(1-p)$ that grain remains at vertex $i$. The coin flips are independent for each edge, and independent of other vertex topplings. We thus define the \emph{stochastic toppling} operator at vertex $i$ by:

\beq\label{eq:part_toppling}
\ST_i(c) := c + \sum_{j: \{i,j\} \in E} \left( \mathds{1}_{B_j = 1} \left( \alpha_j - \alpha_i \right) \right),
\eeq
where the sum is over all vertices $j$ adjacent to $i$, counted with multiplicity, and the $(B_j)$ are i.i.d. Bernoulli random variables of parameter $p$. Note that if we take $p=1$, we have $\ST_i = \DT_i$ a.s., and the SSM is the same as the ASM in that case, hence we assume $p<1$. In general, the resulting configuration $\ST_i(c)$ is a random configuration.

Performing this (stochastic) toppling may cause other vertices to become unstable, and we topple these in turn. 
In~\cite[Theorem 2.2]{CMS} it is shown that, as in the ASM, starting from some unstable configuration $c$ and successively toppling unstable vertices, we eventually reach a (random) stable configuration $c'$.
Moreover, this configuration $c'$ does not depend on the order in which vertices are toppled. 
We write $c' = \StoStab(c)$ and call it the \emph{stochastic stabilisation} of $c$.

As for the ASM, we can define a Markov chain for the SSM on the set of stable configurations $\Stable{G}$.
Let $\mu=(\mu_1,\dots,\mu_n)$ be a probability distribution on $[n]$ such that $\mu_i>0$ for all $i \in [n]$. 
At each step of the Markov chain we add a grain at the vertex $i$ with probability $\mu_i$ and (stochastically) stabilise the resulting configuration.
We call recurrent states for this Markov chain \emph{stochastically recurrent} (SR) and denote their set $\StoRec{G}$. Note that, once again, the Markov chain is irreducible, and the maximal stable configuraiton $\maxc$ is recurrent.

We now give three equivalent characterisations of SR states in terms of grain additions and stabilisations, compatible orientations, and forbidden subconfigurations. These characterisations closely mirror the first three in Theorem~\ref{thm:DR_general} for DR states. In the case of general graphs $G$, there is as yet no known equivalent to Dhar's burning algorithm. In Section~\ref{subsec:SSM_complete}, we describe a stochastic burning algorithm in the case of complete graphs.

\begin{theorem}\label{thm:SR_general}
Let $c = (c_1,\ldots,c_n) \in \Stable{G}$ be a \emph{stable} configuration on the graph $G$. Then $c$ is SR if, and only if, one of the following three equivalent conditions holds.
\begin{enumerate}[label=(\roman*), itemsep=1pt, topsep = 5pt]
\item There exists a configuration $d \in \Config{G}$ such that $\StoStab(\maxc + d) = c$ with positive probability.
\item There exists a \emph{sink-rooted} orientation $\OO$ compatible with $c$.
\item For any subset $A \subseteq [n]$, we have $c(A) \geq \vert E(G[A]) \vert$, where $c(A) := \sum\limits_{i \in A} c_i$ is the total number of grains in $A$, and $\vert E(G[A]) \vert$ is the number of edges of $G$ with both endpoints in $A$ (i.e. the number of edges of the induced subgraph $G[A]$).
\end{enumerate}
\end{theorem}

\begin{proof}

That $c \in \Stable{G}$ is SR if, and only if, Condition~(i) holds follows as in the DR case from being able to put any grain additions first in a sequence of grain additions and topplings. The equivalence with Condition~(ii), was stated in equivalent, though slightly different, form in~\cite[Theorem 3.4]{CMS}. The equivalence of Conditions~(ii) and (iii) is more-or-less part of graph theory ``folklore''. We give a brief proof here for completeness. 

Suppose that there exists a sink-rooted orientation $\OO$ compatible with $c$, and let $A \subseteq [n]$. By summing Inequality~\eqref{eq:comp_orient_config} over all vertices in $A$, we get $c(A) \geq \sum\limits_{i \in A} \In{i} \geq \vert E(G[A]) \vert$. For the right-hand inequality, note that all edges between two vertices in $A$ are counted exactly once in the sum of in-degrees, with this sum also possibly including some edges directed $i \olefta j$ with $i \in A, j \notin A$. This shows that (ii) $\Rightarrow$ (iii).

For the converse, let $c$ be a configuration on $G$ which satisfies Condition~(iii). We proceed by induction on $m = \vert E(G) \vert - \dgr{0}$. If $m = 0$ the result is trivial. Otherwise, let $e = (i, j)$ be some edge of $G$ between two non-sink vertices $i, j \in [n]$. Consider the graph $G'$ with edge $e$ removed, and the configurations $c' := c - \alpha_i$ and $c'' := c - \alpha_j$. We claim that at least one of the configurations $c'$ and $c''$ satisfies Condition~(iii) on the graph $G'$. Suppose for now this claim proved, and without loss of generality assume the condition holds for $c'$. By induction we can find an orientation $\OO'$ of $G'$ compatible with $c'$, and taking $\OO$ to be the orientation $\OO'$ together with the edge $i \olefta j$ yields an orientation $\OO$ of $G$ which is compatible with $c$, as desired.

It therefore remains to prove the previous claim. Seeking contradiction, we assume that Condition~(iii) holds for neither $c'$ nor $c''$. Then there exists some vertex subsets $A_i$ and $A_j$ such that $c'(A_i) < \vert E(G'[A_i]) \vert$ and $c''(A_j) < \vert E(G'[A_j])$. Since Condition~(iii) holds for $c$ in $G$, we must have $i \in A_i$ and $j \in A_j$. Now suppose that $j \in A_i$. We have $c(A_i) = c'(A_i) + 1 < \vert E(G'[A_i]) \vert + 1 = \vert E(G[A_i])$, a contradiction. This means that $j \notin A_i$, and therefore $\vert E(G[A_i]) \vert = \vert E(G'[A_i]) \vert$. In particular, we get $\vert E(G[A_i]) \vert \leq c(A_i) = c'(A_i) + 1 < \vert E(G[A_i]) + 1$, so that $c(A_i) = \vert E(G[A_i])$. By symmetry, we also have $c(A_j) = E(G[A_j])$. Finally, set $A := A_i \cup A_j$, and $A' = A_i \cap A_j$, and let $k$ be the number of edges $(i', j')$ of $G$ such that $i' \in A_i \setminus A'$, $j' \in A_j \setminus A'$. Note that these include at least the edge $(i,j)$, so that $k \geq 1$, and that by construction we have $\vert E(G[A]) \vert = \vert E(G[A_i]) \vert + \vert E(G[A_j]) \vert - \vert E(G[A']) \vert + k$. We get:
\begin{align*}
c(A) & = c(A_i) + c(A_j) - c(A') \\
 & \leq c(A_i) + c(A_j) - \vert E(G[A']) \vert, \qquad \text{applying Condition~(iii) to } A' \\
 & = \vert E(G[A_i]) \vert + \vert E(G[A_j]) \vert - \vert E(G[A']) \vert \\
 & < \vert E(G[A_i]) \vert + \vert E(G[A_j]) \vert - \vert E(G[A']) \vert + 1\\
 & \leq \vert E(G[A_i]) \vert + \vert E(G[A_j]) \vert - \vert E(G[A']) \vert + k = \vert E(G[A]) \vert.
\end{align*}
This contradicts Condition~(iii) for $c$, thus completing the proof.

\end{proof}

\begin{remark}

Conditions~(i) or (ii) of Theorems~\ref{thm:DR_general} and \ref{thm:SR_general} immediately imply the inclusion $\DetRec{G} \subseteq \StoRec{G}$. In general, the converse doesn't hold. Consider the graph in Figure~\ref{fig:sto_non_det} below, with the configuration $c = (1,1,1)$ (the sink is the black square vertex). This configuration is SR since the orientation exhibited on the figure is compatible with $c$, but is not DR as the set of the three non-sink vertices forms a forbidden subconfiguration for Condition~(iii) of Theorem~\ref{thm:DR_general}. Proposition~\ref{pro:DR_SR_equal} will give straightforward, but useful, conditions under which these two sets are the same.
\begin{figure}[ht]
  \centering
  \begin{tikzpicture}[
     scale=0.8,
     pile/.style={very thick, ->, >=stealth'},
     circle/.style={very thick, fill=white} 
     ]
  \path [name path=a] (0,0)--(0,2);
  \path [name path=b] (0,2)--(-1.5,4);
  \path [name path=c] [very thick, ->] (-1.5,4)--(1.5,4);
  \path [name path=d] [very thick, ->] (1.5,4)--(0,2);
  \draw [name path=1, circle] (0,2) circle [radius = 0.5];
  \draw [name path=2, circle] (-1.5,4) circle [radius = 0.5];
  \draw [name path=3, circle] (1.5,4) circle [radius = 0.5];
  \draw [fill=black] (-0.2,-0.4) rectangle (0.2,0);
  \path [name intersections={of=a and 1,by=i1}];
  \path [name intersections={of=b and 1,by=i2}];
  \path [name intersections={of=b and 2,by=i3}];
  \path [name intersections={of=c and 2,by=i4}];
  \path [name intersections={of=c and 3,by=i5}];
  \path [name intersections={of=d and 3,by=i6}];
  \path [name intersections={of=d and 1,by=i7}];
  \draw [pile] (i1)--(0,0);
  \draw [pile] (i2)--(i3);
  \draw [pile] (i4)--(i5);
  \draw [pile] (i6)--(i7);
  \node at (0,2) {$1$};
  \node at (-1.5,4) {$1$};
  \node at (1.5,4) {$1$};
  \end{tikzpicture}
  \caption{Example of a configuration which is SR but not DR.\label{fig:sto_non_det}}
\end{figure}
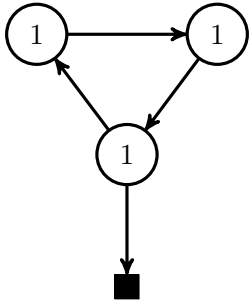

\end{remark}

\begin{definition}\label{def:strong_stable}
Let $c \in \Config{G}$. We say that $c$ is \emph{strongly stable} if for all $i \in [n]$, we have $c_i < \dgr{i} - 1$. In other words, we can add a grain to each vertex of $[n]$ in the configuration $c$, and the resulting configuration will still be stable.
\end{definition}

\begin{remark}\label{rem:strong_stable_can_be_SR}
Suppose that the graph $G$ is simple, i.e.\ has no multiple edges. By Dhar's burning criterion, a strongly stable state can never be DR, since after burning the sink no further vertices can be burned. However, it is possible for a strongly stable state to be SR, as shown in Figure~\ref{fig:superstab_SR} below, which exhibits a strongly stable state on $K_3^0$ with a corresponding compatible orientation (the sink is the central black square).
\begin{figure}[ht]
  \centering
  \begin{tikzpicture}[
     scale=0.8,
     pile/.style={very thick, ->, >=stealth'},
     circle/.style={very thick, fill=white} 
     ]
  \path [name path=a] (0,-0.4)--(0,2); 
  \path [name path=b] (0,-0.4)--(-2,-2); 
  \path [name path=c] (0,-0.4)--(2,-2); 
  \path [name path=d] (0,2)--(-2,-2); 
  \path [name path=e] (-2,-2)--(2,-2); 
  \path [name path=f] (2,-2)--(0,2); 
  \draw [name path=1, circle] (0,2) circle [radius = 0.5];
  \draw [name path=2, circle] (-2,-2) circle [radius = 0.5];
  \draw [name path=3, circle] (2,-2) circle [radius = 0.5];
  \draw [name path=0, fill=black] (-0.2,-0.6) rectangle (0.2,-0.2);
  \path [name intersections={of=a and 0,by=i01}];
  \path [name intersections={of=a and 1,by=i10}];
  \path [name intersections={of=b and 0,by=i02}];
  \path [name intersections={of=b and 2,by=i20}];
  \path [name intersections={of=c and 0,by=i03}];
  \path [name intersections={of=c and 3,by=i30}];
  \path [name intersections={of=d and 1,by=i12}];
  \path [name intersections={of=d and 2,by=i21}];
  \path [name intersections={of=e and 2,by=i23}];
  \path [name intersections={of=e and 3,by=i32}];
  \path [name intersections={of=f and 3,by=i31}];
  \path [name intersections={of=f and 1,by=i13}];
  \draw [pile] (i10)--(i01);
  \draw [pile] (i20)--(i02);
  \draw [pile] (i30)--(i03);
  \draw [pile] (i12)--(i21);
  \draw [pile] (i23)--(i32);
  \draw [pile] (i31)--(i13);
  \node at (0,2) {$1$};
  \node at (-2,-2) {$1$};
  \node at (2,-2) {$1$};
  \end{tikzpicture}
  \caption{Example of a configuration which is strongly stable and SR.\label{fig:superstab_SR}}
\end{figure}
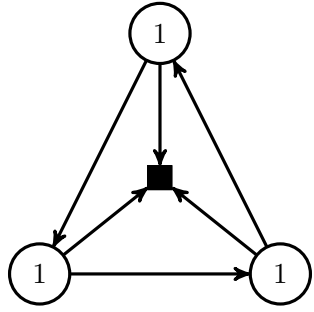
\end{remark}

\subsection{Minimal recurrent configurations}\label{subsec:minrec}

In this part we mention briefly a few results about \emph{minimal} recurrent configurations. The results are nearly identical for minimal SR or DR states, so we make a common section for both models for the sake of brevity and to avoid too much repetition, and talk simply of recurrent states.

First, observe that if $c$ is a recurrent state, by summing Inequality~\eqref{eq:comp_orient_config} over all vertices $i \in [n]$, we get:
\beq\label{eq:total_grains_ineq}
\sum\limits_{i=1}^n c_i \geq \vert E \vert - \dgr{0}.
\eeq
This naturally leads to the definition of the \emph{level} statistic of a recurrent configuration: 
\beq\label{eq:level_def}
\level(c) := \left( \sum\limits_{i=1}^n c_i \right) + \dgr{0} - \vert E \vert.
\eeq
By also noting that a recurrent configuration is stable, so $c_i \leq \dgr{i} - 1$ for all $i \in [n]$, we get the bounds:
\beq\label{eq:bounds_level}
0 \leq \level(c) \leq \vert E \vert - n,
\eeq
which hold for any recurrent configuration $c$.

There is a natural partial order $\preceq$ on the set of configurations. For $c,c' \in \Config{G}$, we define $c \preceq c'$ if for all $i \in [n]$, $c_i \leq c'_i$. From the Markov chain definition, if $c\preceq c'$ with $c$ a recurrent state, and $c'$ is stable, then $c'$ is also recurrent. We say that $c$ is \emph{minimal recurrent} if $c$ is recurrent and minimal for the partial order $\preceq$. 

\begin{theorem}\label{thm:char_minrec_general_orientation}
Let $c \in \Stable{G}$ be a stable configuration on $G$. Then $c$ is minimal SR, resp.\ minimal DR, if, and only if, there exists a sink-rooted, resp.\ acyclic sink-rooted, orientation $\OO$ of $G$, such that:
\beq\label{eq:char_minrec_general_orientation}
\forall i \in [n], \, c_i = \In{i}.
\eeq
Moreover, such an orientation is unique up to score-equivalence, resp.\ unique.
\end{theorem}

\begin{proof}[Sketch of proof]
Theorem~\ref{thm:char_minrec_general_orientation} was proved in the (minimal) DR case in~\cite{Schulz}. The proof in the SR case is identical. If $\OO$ is a sink-rooted orientation, compatible with a SR state $c$, and there exists $i \in [n]$ such that $c_i > \In{i}$, then $c':=c-\alpha_i$ is still compatible with $\OO$, and so is SR by Theorem~\ref{thm:SR_general}, which implies that $c$ is not minimal. 

Conversely, if $c$ is SR and there exists a sink-rooted orientation $\OO$ which satisfies Equation~\eqref{eq:char_minrec_general_orientation}, then by summation we have $\level(c) = 0$. If $c$ were not minimal, there would exist $c' \in \StoRec{G}$ such that $c' \prec c$, and by summation we would have $\level(c') < \level(c) = 0$, which is impossible by the left-hand side of Inequality~\eqref{eq:bounds_level}, so $c$ must be minimal, as desired.

The uniqueness of $\OO$ up to score-equivalence follows immediately from the characterisation. In the DR case, as noted in Section~\ref{subsec:orient}, uniqueness up to score equivalence of acyclic orientations is simply uniqueness.
\end{proof}

The proof above suggests a link between minimal recurrent states and those with level equal to $0$. This is indeed the case.

\begin{proposition}\label{pro:char_minrec_general_level}
Let $c$ be a recurrent configuration. Then $c$ is minimal recurrent if, and only if, $\level(c) = 0$.
\end{proposition}

This result is often implicit in the literature, particularly when considering the \emph{level polynomial} of a graph which counts recurrent configurations according to their level. We sketch a brief proof here for completeness.

\begin{proof}[Sketch of proof]
If $c$ is recurrent and $\level(c) = 0$, then $c$ is minimal recurrent. Otherwise there would exist a recurrent state $c'$ such that $c' \prec c$, and by summation we would $\level(c') < \level(c) = 0$, which is impossible by the left-hand side of Inequality~\eqref{eq:bounds_level}.

Conversely, if $c$ is minimal recurrent, there exists a sink-rooted (acyclic in DR case) orientation $\OO$ satisfying Equation~\eqref{eq:char_minrec_general_orientation}, and by summation we get $\level(c) = 0$ as desired.
\end{proof}

We end this section by stating conditions under which all SR states are also DR (we know that the converse is always true).

\begin{proposition}\label{pro:DR_SR_equal}
Let $G$ be a graph with vertex set $[n]_0$ and edge set $E$. The following statements are equivalent.
\begin{enumerate}[label=(\roman*), itemsep=1pt, topsep = 5pt]
\item All SR states on $G$ are also DR, i.e. $\StoRec{G} = \DetRec{G}$.
\item All \emph{minimal} SR states on $G$ are also DR.
\item The graph $G' := G\big[ [n] \big]$ obtained by deleting the sink and all incident edges in $G$ is a forest.
\end{enumerate}
\end{proposition}

\begin{proof}
Trivially, (i) implies (ii). It is also straightforward to show that (iii) implies (i). Let $c$ be a SR state. By Theorem~\ref{thm:SR_general} there exists a sink-rooted orientation $\OO$ of $G$ that is compatible with $c$. But if $G'$ is a forest, it contains no cycles, so that any sink-rooted orientation of $G$ is acyclic (there can be no directed cycles containing the sink, since it is a root of the orientation). Therefore $\OO$ is an acyclic sink-rooted orientation of $G$, compatible with $c$, and so $c$ is DR by Theorem~\ref{thm:DR_general}.

It therefore remains to show that (ii) implies (iii), which we do by contraposition. Suppose that $G'$ contains a cycle $(i_1, \ldots i_k, i_1)$. Let $G''$ be the graph $G$ with all edges of this cycle removed, and $\OO'$ be a sink-rooted orientation of $G''$. Now let $\OO$ be the sink-rooted orientation of $G$ obtained by adding the directed cycle $i_1 \orighta i_2 \orighta \cdots \orighta i_k \orighta i_1$ to $\OO'$, and define $c$ to be the minimal SR state compatible with $\OO$. If $c$ were DR, there would exist an acyclic orientation that is compatible with $c$, and therefore score-equivalent to $\OO$. But by Proposition~\ref{pro:equiv_orient_cycle_flip} this is impossible:  one can't go from an acyclic orientation to one containing a directed cycle through a series of cycle-flips. Therefore $c$ is not DR, as desired.
\end{proof}

\subsection{Parking functions}\label{subsec:parking_functions}

\begin{definition}\label{def:parking_function}
Let $p=(p_1,\ldots,p_n) \in \N^n$ be a tuple of positive integers, and $\inc{p}$ its non-decreasing rearrangement. We say that $p$ is a \emph{parking function} if
\beq\label{eq:parking_function_def}
\forall i \in [n], \inc{p}_i \leq i.
\eeq
The set of all $n$-parking functions is denoted $\PF{n}$.
\end{definition}

\begin{remark}
The terminology ``parking function'' comes from the following observation. Suppose we have $n$ cars trying to park in $n$ spaces, cars and spaces both labeled $1,\ldots,n$. Each car has a preferred parking spot $p_i$. The cars enter the car park in order $1,\ldots,n$ and for each $i$ the car $i$ parks in the first available spot $q \geq p_i$ (if no such spot exists, the car cannot park). Then all $n$ cars can park if, and only if, $p=(p_1,\ldots,p_n)$ is a parking function.
\end{remark}

There is a rich literature on the study of parking functions. We refer interested readers to the excellent survey by C. Yan~\cite{Yan}. In Theorem~\ref{thm:detrec_pf} we recall a straightforward bijection between $n$-parking functions and DR states of the ASM on the complete graph $K_n^0$.

\subsection{Polytopes}\label{subsec:polytopes}

A polytope is the convex hull of a finite set of points $S \subset \R^n$ for some $n$. In other words, it is the set of all points of the form $\sum\limits_{s \in S} \lambda_s s$, where $0 \leq \lambda_s \leq 1$ for all $s \in S$ and  $\sum\limits_{s \in S} \lambda_s = 1$. In this paper, we will be interested in three polytopes in particular.
\begin{itemize}
\item The regular $n$-permutohedron. This is the polytope of the set $\mathcal{S}_n$ of permutations of length $n$. We denote it $\PermPoly{n}$.
\item The $n$-dimensional parking function polytope. This is the polytope of the set of $n$-parking functions $\PF{n}$. We denote it $\PFPoly{n}$.
\item The $n$-dimensional DR polytope. This is the polytope of the set of DR states on the complete graph $K_n^0$. We denote it $\DRPoly{n}$.
\end{itemize}

Polytopes, and in particular permutohedrons, have been well-studied in the combinatorics literature (see \cite{Post} and references therein). One element of interest when studying a polytope combinatorially (usually in this case $S \subset \Z^n$) is to look at its set of integer lattice points, i.e.\ the set of points in the polytope with integer coordinates. In general these sets are not easy to compute, let alone obtain an explicit formula for. One of the main results in our paper (Theorem~\ref{thm:SR_DRPoly}) states that the set of SR states on the complete graph $K_n^0$ is the set of integer lattice points in the $n$-dimensional DR polytope $\DRPoly{n}$.


\section{The sandpile model(s) on complete graphs}\label{sec:sandpile_complete}

In this section, we study the ASM and SSM on the complete graphs $K_n^0$. Recall that the complete graph $K_n^0$ is the graph with vertex set $[n]_0 = \{0,1,\cdots,n\}$ where there is one edge between any pair of distinct vertices. Because the sink is connected to all the non-sink vertices, there is an obvious one-to-one correspondence between sink-rooted orientations of $K_n^0$ and orientations of $K_n$, where $K_n$ is the complete graph with vertex set $[n] = \{1,\cdots,n\}$, i.e. the graph $K_n^0$ with sink removed.

This means that we can essentially ignore the role the sink plays in characterising recurrent states for the sandpile model on $K_n^0$. Therefore, to lighten notation, we will slightly abuse notation and talk of the sandpile model on, and orientations of, $K_n$. We will write $\Confign$ for the set of configurations on $K_n$, with similar notation for $\Stablen$, $\DetRecn$, and $\StoRecn$.

The following result is an immediate consequence of the symmetry of the complete graph, which we state here as it is useful at various points in the paper.

\begin{proposition}\label{pro:comp_inc_sym}
Let $c = (c_1,\cdots,c_n) \in \Confign$ be a configuration on $K_n$. Then $c$ is recurrent (DR or SR) if, and only if, its non-decreasing rearrangement $\inc{c}$ is recurrent.
\end{proposition}

\subsection{The ASM on $K_n$: parking functions and the burning algorithm}\label{subsec:ASM_complete}

In this part, we recall the bijection from the set of DR states on $K_n$ and the set of $n$-parking functions from the seminal work by Cori and Rossin~\cite{CR}.

\begin{theorem}\label{thm:detrec_pf}
Let $c = (c_1,\cdots,c_n) \in \Confign$ be a configuration of the ASM on $K_n$. Define $p = (p_1,\cdots,p_n) := (n-c_1,\cdots,n-c_n)$ to be the n-complement of $c$ (we write $p = n-c$ for short). Then $c$ is DR if, and only if, $p$ is a parking function. Thus the map $c \mapsto n-c$ defines a bijection from $\DetRecn$ to $\PF{n}$.
\end{theorem}

We get the immediate following corollary, which can be thought of as a re-writing of Dhar's burning criterion for complete graphs.

\begin{corollary}\label{cor:detrec_comp_burn}
Let $c = (c_1,\cdots,c_n) \in \Stablen$ be a stable configuration on $K_n$. Then $c$ is DR if, and only if, for all $i \in [n]$, we have $\inc{c}_i \geq i-1$, where $\inc{c} = (\inc{c}_1, \cdots, \inc{c}_n)$ is the non-decreasing rearrangement of $c$.
\end{corollary}

\subsection{The SSM on $K_n$: a stochastic burning algorithm}\label{subsec:SSM_complete}

In this part, we study the SSM on the complete graphs $K_n$. First, we restate the characterisation result from Theorem~\ref{thm:SR_general}, Condition~(ii), in the complete graph case.

\begin{theorem}\label{thm:SR_complete_char}
Let $c = (c_1,\ldots,c_n) \in \Stablen$ be a stable configuration on $K_n$. Then $c$ is SR if, and only if:
\beq\label{eq:SR_complete_char}
\forall A \subseteq [n], \, c(A) \geq \binom{\vert A \vert}{2}.
\eeq
\end{theorem}

A more convenient characterisation that will allow us to exhibit the \emph{stochastic burning algorithm} (Algorithm~\ref{algo:sto_burning}) is the following.

\begin{proposition}\label{pro:SR_complete_inc}
Let $c = (c_1,\ldots,c_n) \in \Stablen$ be a stable configuration on $K_n$, and $\inc{c}$ its non-decreasing rearrangement. Then $c$ is SR if, and only if:
\beq\label{eq:SR_complete_inc}
\forall\, i \in [n], \, \sum\limits_{j=1}^i \inc{c}_j = \inc{c}([i]) \geq \binom{i}{2}.
\eeq
\end{proposition}

\begin{proof}
If $c$ is SR, then so is $\inc{c}$ by Proposition~\ref{pro:comp_inc_sym}. Applying Equation~\eqref{eq:SR_complete_char} to $\inc{c}$ and subsets $[i]$ immediately gives Equation~\eqref{eq:SR_complete_inc}.

Conversely, suppose that $c$ is such that $\inc{c}$ satisfies Equation~\eqref{eq:SR_complete_inc}. To simplify notation, by Proposition~\ref{pro:comp_inc_sym} we may assume that $c = (c_1, \ldots, c_n) = \inc{c}$. Let $A \subseteq [n]$ and write $A = \{i_1, \ldots, i_k\} $ with $i_1 < \cdots < i_k$. By construction, we have $i_j \geq j$ for all $j \in \{1, \ldots, k\}$, and thus $c_{i_j} \geq c_j$ since $c$ is assumed to be non-decreasing. This implies that
$$ \begin{array}{r c l}
c(A) & = & \sum\limits_{j=1}^k c_{i_j}\\
     & \geq & \sum\limits_{j=1}^k c_{j}\\[2pt]
     & \geq & \binom{k}{2}\\[5pt]
     & = & \binom{\vert A \vert}{2},
\end{array} $$
where we applied Equation~\eqref{eq:SR_complete_inc} in the third line. This shows that $c$ is SR by Theorem~\ref{thm:SR_complete_char}.
\end{proof}

Proposition~\ref{pro:SR_complete_inc}  leads us to define the stochastic burning algorithm on the complete graph $K_n$.

\begin{algorithm}[Stochastic burning algorithm]\label{algo:sto_burning}
Input: $c \in \Stablen$, a stable configuration on $K_n$.
\begin{enumerate}
\item \textbf{Step 1}. Obtain the non-decreasing rearrangement $\inc{c}$ of $c$. Define $\mathrm{Sum} := \sum\limits_{i=1}^n c_i$ and $\mathrm{Target} := \binom{n}{2}$. Initialise $k:=n$.
\item \textbf{Step 2}. While $\mathrm{Sum} \geq \mathrm{Target}$ and $k>0$, do:
  \begin{itemize}[topsep=0pt, noitemsep]
  \item $\mathrm{Sum} \leftarrow \mathrm{Sum} - \inc{c}_k$,
  \item $\mathrm{Target} \leftarrow \mathrm{Target} - (k-1)$
  \item $k \leftarrow k-1$.
  \end{itemize}
 This step is equivalent to burning the vertex with the $k$-th smallest number of grains in $c$.
\item \textbf{Step 3}. Output $k$, the number of unburned vertices.
\end{enumerate}
\end{algorithm}

\begin{theorem}\label{thm:sto_burning}
Let $c \in \Stablen$ be a stable configuration on $K_n$. Then $c$ is SR if, and only if, all vertices are burned in the stochastic burning algorithm, that is, Algorithm~\ref{algo:sto_burning} outputs $k=0$.\\
Moreover, the stochastic burning algorithm runs in linear $O(n)$ time.
\end{theorem}

\begin{proof}
The characterisation follows from Proposition~\ref{pro:SR_complete_inc} and the straightforward observation that $\binom{k}{2} - (k-1) = \binom{k-1}{2}$. For the complexity, we note that since the configuration $c$ is bounded above by $n$ (we have $c_i \leq n-1$ for all $i \in [n]$), sorting can be done in linear time using the $\mathrm{CountSort}$ algorithm (also known as \emph{sort by values}, see e.g.~\cite[proof of Proposition~13]{CLB_Baker}). In brief, this algorithm first computes an auxiliary array $a = (a_0, \dots, a_{n-1})$ where $a_j = \left\vert \{i \in [n]; c_i = j \} \right\vert$ for all $j$. This can be done straightforwardly in linear time by simply looping over $c$ and incrementing $a_{c_i}$ by one at each step. We then recover $\inc{c}$ from $a$ in linear time by looping over $a$ and appending $a_j$ times the value $j$ to $\inc{c}$ at each step (since $\sum a_j = n$ this is indeed linear). Finally, once the sorting is done, the algorithm itself requires only a linear number of calculations.
\end{proof}

\begin{remark}\label{rem:sto_burning_only_complete}
The stochastic burning algorithm only works in the complete graph case, where this high degree of symmetry allows us to reduce configurations to their non-decreasing rearrangements. A natural attempt at generalisation would be to successively burn maximal vertices. That is, we check the condition $c(A) \geq \big\vert E(G(A)) \big\vert$ (Condition~(iii) from Theorem~\ref{thm:SR_general}) for subsets $A$ obtained by successively burning vertices with the maximal number of grains. This however fails fairly obviously, since we can ``artificially'' inflate the number of grains at a given vertex $v$ by simply connecting an arbitrary number of new vertices to $v$ (these new vertices all have degree $1$). 

A less naive attempt would be to instead successively burn vertices with minimal \emph{lacking} number, where the lacking number of a vertex is equal to its degree minus its number of grains. However, this also fails. For a counter-example, consider the configuration in Figure~\ref{fig:counter_sto_burn} below. For convenience, we don't represent the sink, but we may assume that it is connected to each vertex by a single edge. The configuration, read from left-to-right, is $(0, 5, 5, 5)$, and the lacking numbers are $(7-0, 7-5, 8-5, 8-5) = (7, 2, 3, 3)$. The two left-most vertices (in red) form a forbidden subconfiguration in terms of Theorem~\ref{thm:SR_general}, Condition~(iii), since there are $5$ grains in total but $6$ edges between the two. This implies that the configuration is not SR. One can check that this is the only such forbidden sub-configuration for the SSM. But now, notice that the vertex with minimal lacking number is part of this forbidden subconfiguration. As such, if we were to burn vertices successively in (increasing) order of their lacking number, this would be the first vertex to be burned, and so there is no possibility of the algorithm terminating on the forbidden subconfiguration: it would simply burn through it!

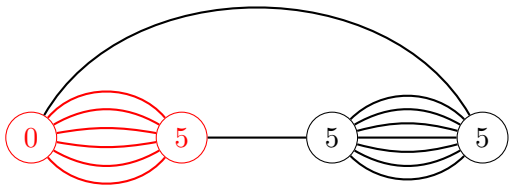
\begin{figure}[ht]
 \centering
 \begin{tikzpicture}
 \node [draw, circle, red] (1) at (0,0) {$0$};
 \node [draw, circle, red] (2) at (2,0) {$5$};
 \node [draw, circle] (3) at (4,0) {$5$};
 \node [draw, circle] (4) at (6,0) {$5$};
 \draw [thick] (2)--(3)--(4);
 \draw [thick, red] (1) [out=10, in=170] to (2);
 \draw [thick, red] (1) [out=-10, in=-170] to (2);
 \draw [thick, red] (1) [out=30, in=150] to (2);
 \draw [thick, red] (1) [out=-30, in=-150] to (2);
 \draw [thick, red] (1) [out=50, in=130] to (2);
 \draw [thick, red] (1) [out=-50, in=-130] to (2);
 \draw [thick] (3) [out=15, in=165] to (4);
 \draw [thick] (3) [out=-15, in=-165] to (4);
 \draw [thick] (3) [out=30, in=150] to (4);
 \draw [thick] (3) [out=-30, in=-150] to (4);
 \draw [thick] (3) [out=45, in=135] to (4);
 \draw [thick] (3) [out=-45, in=-135] to (4);
 \draw [thick] (1) [out=60, in=120] to (4);
 \end{tikzpicture}
 \caption{A counter-example to a stochastic burning algorithm which successively burns vertices with minimal lacking number. The sub-configuration on the two left-most vertices (in red) is a forbidden sub-configuration, but all vertices would be burned. \label{fig:counter_sto_burn}}
\end{figure}
\end{remark}

The question of an (efficient) stochastic burning algorithm on general graphs remains a significant open problem. In fact, it remains an open problem to improve the existing (trivial) complexity bound of $O\left( 2^{\vert V \vert} \right)$ resulting from checking Condition~(iii) of Theorem~\ref{thm:SR_general} for all subsets $A \subset [n]$.

\subsection{Minimal SR states on $K_n$: tournaments and spanning forests}\label{subsec:SSM_complete_minrec}

We now briefly study minimal SR states on the complete graph. An $n$-\emph{tournament} is an orientation of the complete graph $K_n$ for some $n$. From Theorem~\ref{thm:char_minrec_general_orientation} and our remarks at the beginning of Section~\ref{sec:sandpile_complete}, we know that there is a bijection between minimal SR states on $K_n$ and the score-equivalent classes of $n$-tournaments. 

A \emph{spanning forest} of a graph $G = (V,E)$ is a cycle-free subgraph $F = (V,E)$ (or equivalently, a collection of disjoint trees such that the union of their vertices is equal to $V$). The following result has been known for some time in graph theory, see for instance~\cite{KW} for a bijective proof.

\begin{theorem}\label{thm:score_vect_spanning_forests}
The number of score-equivalent orientations of a graph $G$ is equal to its number of spanning forests.
\end{theorem}

This immediately gives another enumeration of minimal SR states.

\begin{proposition}\label{pro:minSR_spanning_forests}
The number of minimal SR states on $K_n$ is equal to the number of labelled spanning forests on $n$ nodes.
\end{proposition}

\begin{remark}\label{rem:minSR_spanning_forests}
In fact, Proposition~\ref{pro:minSR_spanning_forests} holds in a more general setting than that of the complete graphs $K_n$. It holds in any graph where the sink plays no role in enumerating the minimal SR states. Thus, for any graph $G$ where every non-sink vertex has at least one edge to the sink $0$, the number of minimal SR states on $G$ is equal to the number of spanning forests of $G \setminus \{0\}$ (the graph $G$ with the sink removed). Conversely, if this equality holds, then every non-sink vertex in $G$ must have at least one edge to the sink.
\end{remark}

Because of the symmetries of $K_n$, it is natural as we have seen to consider non-decreasing rearrangements of recurrent states. We can do this for minimal SR states, i.e. consider configurations $c = (c_1, \ldots, c_n)$ that are minimal SR, and weakly increasing ($c_1 \leq c_2 \leq \cdots \leq c_n$). In the context of complete graphs, these are known in the literature as \emph{tournament scores} (a tournament score is the weakly increasing re-arrangement of the in-degree sequence of a tournament). The characterisation of a tournament score was originally given by Landau~\cite{Land} and is essentially Proposition~\ref{pro:SR_complete_inc} to which are added stability and minimality conditions. In recent work~\cite{CDFS}, Claesson et al.\ proved an elegant enumeration formula for the number of tournament scores, and their result yields an algorithm that generates these numbers in $O\!\left( n^2 \right)$ time ($n$ is the size of the tournament). Thus, we can generate the number of non-decreasing minimal SR states on $K_n$ in quadratic $O\!\left( n^2 \right)$ time.


\section{Characterisation of $\StoRecn$}\label{sec:main_result}

In this section we state and prove our main characterisation theorem: that the SR states on $K_n$ are given by the integer lattice points of the DR polytope $\DRPoly{n}$. We also give a few consequences of this result.

\subsection{Statement of the result and consequences}\label{subsec:statement_result}

Recall that the $n$-dimensional DR polytope $\DRPoly{n}$ is the polytope of the set $\DetRecn$ of DR states on $K_n$, and that the $n$-dimensional parking function polytope $\PFPoly{n}$ is the polytope of the set $\PF{n}$ of $n$-parking functions. Note that the bijection from Theorem~\ref{thm:detrec_pf} between the sets $\DetRecn$ and $\PF{n}$ extends immediately to a bijection between integer lattice points in the polytopes $\DRPoly{n}$ and $\PFPoly{n}$. The main result of this section is the following.

\begin{theorem}\label{thm:SR_DRPoly}
The set $\StoRecn$ of SR states for the SSM on the complete graph $K_n$ is the set of integer lattice points (points with integer coordinates) in the DR polytope $\DRPoly{n}$. In other words, a configuration $c=(c_1,\ldots,c_n) \in \Confign$ is SR if, and only if, there exist DR states $c^{(1)},\ldots,c^{(k)} \in \DetRecn$, and scalars $\lambda_1,\ldots,\lambda_k \in [0,1]$ with $\sum\limits_{i=1}^k \lambda_i = 1$, such that $c = \sum\limits_{i=1}^k \lambda_i c^{(i)}$, where the sum and scalar multiplication are the usual pointwise operations in $\R^n$.
\end{theorem}

From the introductory remark of this section, we immediately get the following.

\begin{corollary}\label{cor:SR_PFPoly}
The number of SR states for the SSM on the complete graph $K_n$ is the number of integer lattice points in the parking function polytope $\PFPoly{n}$.
\end{corollary}

The first values of these numbers for $n=1,2,\ldots$ are: $1, 3, 17, 144, 1623, 22804, \ldots$, given by Sequence~A333331 in~\cite{OEIS}. An explicit, if somewhat complicated, enumerative formula for this sequence is given in~\cite[Theorem~5.1]{AW}. Corollary~\ref{cor:SR_PFPoly} provides a new combinatorial interpretation of these numbers.

In fact (see Remark~\ref{rem:spec_level}), our proof in Section~\ref{subsec:proof_result} implies a stronger result than Theorem~\ref{thm:SR_DRPoly}. More specifically, we will show that the result holds when restricting ourselves to recurrent states of a fixed level, i.e. that the set of SR states on $K_n$ with level $k$ is the set of integer lattice points in the convex hull of DR states on $K_n$ with level $k$, for any $k \leq \binom{n}{2}$. A consequence of this, for the special case $k=0$, gives the following well-known result (see Exercise~4.64(a) in~\cite{Stanley}). Recall that the regular $n$-permutohedron is the polytope of the set of $n$-permutations.

\begin{corollary}\label{cor:PermPoly_forests}
The number of integer points in the regular $n$-permutohedron is the number of labelled spanning forests on $n$ vertices.
\end{corollary}

\begin{proof}
Theorem~\ref{thm:SR_DRPoly} restricted to minimal recurrent states shows that the set of minimal SR states on $K_n$ is the set of integer lattice points in the convex hull of the set of minimal DR states on $K_n$. But minimal DR states correspond to maximal parking functions via the bijection of Theorem~\ref{thm:detrec_pf}. Because maximal parking functions are just permutations, this implies that the number of minimal SR states on $K_n$ is the number of integer lattice points in the regular $n$-permutohedron. The result then follows immediately from Proposition~\ref{pro:minSR_spanning_forests}.
\end{proof}

\begin{remark}\label{rem:DetRecPoly_false_general}
Theorem~\ref{thm:SR_DRPoly} is false for general graphs. Consider the graph of Figure~\ref{fig:sto_non_det} with its SR state $c=(1,1,1)$. All DR states for this graph have two grains at the vertex $1$ connected to the sink (this is easily checked by e.g.\ Dhar's burning algorithm), so $c$ cannot be written as a convex sum of DR states.
\end{remark}

\subsection{Proof of Theorem~\ref{thm:SR_DRPoly}}\label{subsec:proof_result}

We first show that an integer lattice point in the DR polytope $\DRPoly{n}$ is SR. Let $c \in \DRPoly{n} \cap \Z^n$ be such a point. We can write $
c = \sum\limits_{i=1}^k \lambda_i c^{(i)} $,
for some DR states $c^{(1)},\ldots,c^{(k)} \in \DetRecn$, and scalars $\lambda_1,\ldots,\lambda_k \in [0,1]$ such that $\sum\limits_{i=1}^k \lambda_i = 1$. Let $A \subseteq [n]$. We have:
$$ \begin{array}{r c l}
c(A) & = & \sum\limits_{i=1}^k \lambda_i c^{(i)}(A)\\
     & \geq & \sum\limits_{i=1}^k \lambda_i \binom{\vert A \vert}{2} \\
     & = & \binom{\vert A \vert}{2},
\end{array} $$
where the inequality on the second line follows from the inclusion $\DetRecn \subseteq \StoRecn$ and applying Theorem~\ref{thm:SR_complete_char} to each $c^{(i)}$. Since this holds for any subset $A \subseteq [n]$, by Theorem~\ref{thm:SR_complete_char}, we have $c \in \StoRecn$ as desired.

\medskip

We now show the converse, that is 
\beq\label{eq:SR_inc_DRPoly}
\StoRecn \subseteq \DRPoly{n}.
\eeq 
We proceed by strong induction on $n$. For $n=1$ or $n=2$, we have $\StoRecn = \DetRecn$ (in these cases $K_n$ is acyclic, so Proposition~\ref{pro:DR_SR_equal} applies), and the result is trivial. Suppose now that Inclusion~\eqref{eq:SR_inc_DRPoly} holds for all $k<n$ for some $n \geq 3$, and let $c = (c_1,\ldots,c_n) \in \StoRecn$ be a SR state on $K_n$.

We first introduce some notation. Let $\maxi \in [n]$, resp.\ $\nextMax \in [n]$, be such that $c_{\maxi} = \max\limits_{i \in [n]} \{ c_i\}$, resp.\ $c_{\nextMax} = \max\limits_{\substack{ i \in [n] \\ i \neq \maxi }} \{ c_i\}$. In words, $\maxi$ is the label of the vertex with the highest number of grains in the configuration $c$, and $\nextMax$ the label of the vertex with the second-highest number. Our first step is a reduction: we show that it is sufficient to consider the case where $c_{\maxi} = n-1$, with the help of the following lemma.

\begin{lemma}\label{lem:add_grain_to_max}
Let $\maxi$ and $\nextMax$ be defined as above, and assume that $c$ is strongly stable, i.e. $c_{\maxi} < n-1$. Define a configuration $c'$ by: 
$$ c'_i = \begin{cases} 
c_{\maxi} + 1 & \text{if } i = \maxi, \\
c_{\nextMax} - 1 & \text{if } i = \nextMax, \\
c_{i} & \text{otherwise}.
\end{cases}$$
Then $c' \in \StoRecn$.
\end{lemma}

In words, Lemma~\ref{lem:add_grain_to_max} tells us that if $c \in \StoRecn$ is strongly stable, then we can make another SR state from $c$ by moving one grain from vertex $\nextMax$ to vertex $\maxi$.

\begin{proof}
Let $c \in \StoRecn$ such that $c_{\maxi} < n-1$. Note that since $n \geq 3$ we must have $c_{\nextMax} > 0$ since otherwise the set $[n]$ would violate Equation~\eqref{eq:SR_complete_char}. We show that there exists an orientation $\OO$ of $K_n$, compatible with $c$, such that $\maxi \orighta \nextMax$. If such an orientation exists, then let $\OO'$ be the orientation $\OO$ with $\maxi \xleftarrow{\OO'} \nextMax$ and all other edges unchanged. By construction, $\OO'$ is compatible with $c'$, and the lemma is proved. It is thus sufficient to show the existence of such an orientation $\OO$.

Consider an orientation $\OO$ compatible with $c$, and suppose we have $\nextMax \orighta \maxi$ (otherwise there is nothing more to prove). We may also assume that $\In{\nextMax} = c_{\nextMax}$, since otherwise we may simply flip the orientation of $(\nextMax, \maxi)$ without changing the compatibility of the orientation. We will show that we can find a directed cycle of length $3$ or $4$ in $\OO$ which contains the edge $\nextMax \orighta \maxi$ (see Figure~\ref{fig:4-cycle} for how the $4$-cycle is constructed). Flipping this directed cycle then yields the desired orientation.

Since $c_{\maxi} < n-1$ there must exist $i \in [n]$, $i \neq \nextMax$,  such that $\maxi \orighta i$. If $i \orighta \nextMax$, we have a $3$-cycle $\maxi \orighta i \orighta \nextMax \orighta \maxi$. Otherwise consider the set of vertices $j$ such that $j \orighta \nextMax$. By construction there are $c_{\nextMax} > 0$ such vertices. 
And for at least one of these, we must have $i \orighta j$, since otherwise there would be at least $c_{\nextMax} + 1$ incoming edges at $i$ ($c_{\nextMax}$ for those such vertices $j$, plus $1$ for the edge $\maxi \orighta i$). This would imply $c_i \geq \In{i} \geq c_{\nextMax} + 1$, which would contradict the definition of $\nextMax$. Therefore there exists a $4$-cycle $\maxi \orighta i \orighta j \orighta \nextMax \orighta \maxi$.

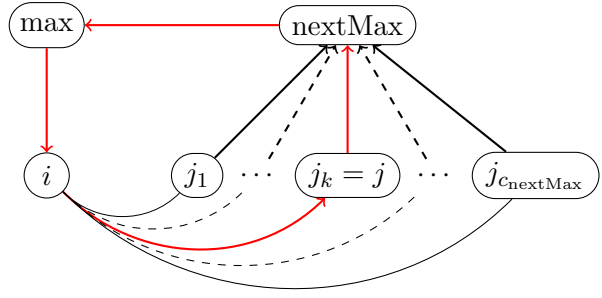
\begin{figure}[ht]
\centering
\begin{tikzpicture}

\node[rounded rectangle, draw=black, minimum size=0.6cm] (max) at (0,0) {$\maxi$};
\node[rounded rectangle, draw=black] (nextMax) at (4,0) {$\nextMax$};
\node[rounded rectangle, draw=black] (j1) at (2,-2) {$j_1$};
\node[rounded rectangle, draw=black] (j) at (4,-2) {$j_k=j$};
\node[rounded rectangle, draw=black] (jlast) at (6.5,-2) {$j_{c_{\nextMax}}$};
\node[xshift=-5] (d1) at (3,-2) {$\cdots$};
\node[xshift=5] (d2) at (5,-2) {$\cdots$};
\node[rounded rectangle, draw=black, minimum size=0.6cm] (i) at (0,-2) {$i$};

\draw[thick, ->,red] (nextMax)--(max);
\draw[thick, ->] (j1)--(nextMax);
\draw[thick, ->,dashed] (d1)--(nextMax);
\draw[thick, ->,red] (j)--(nextMax);
\draw[thick, ->,dashed] (d2)--(nextMax);
\draw[thick, ->] (jlast)--(nextMax);
\draw[thick, ->,red] (max)--(i);
\draw[thick, ->,red] (i) [out=-45, in=-135] to (j);
\draw[dashed] (i) [out=-45, in=-135] to (d1);
\draw[dashed] (i) [out=-45, in=-135] to (d2);
\draw[very thin] (i) [out=-45, in=-135] to (j1);
\draw (i) [out=-45, in=-135] to (jlast);

\end{tikzpicture}

\caption{Illustrating the $4$-cycle construction. Among the edges $(i,j_1),\cdots,\left( i, j_{c_{\nextMax}} \right)$, there must be at least one edge directed $i \orighta j_k$, which forces the necessary directed cycle (in red).
\label{fig:4-cycle}
}

\end{figure}

Thus in all cases we have shown the existence of a directed cycle in $\OO$ containing the edge $\nextMax \orighta \maxi$. Flipping this cycle (and leaving other edges unchanged) gives us the desired compatible orientation $\OO'$ with the edge $\nextMax \xleftarrow{\OO'} \maxi$, which concludes the proof of the lemma.
\end{proof}

We now complete the proof of Theorem~\ref{thm:SR_DRPoly}. Assume first that $c$ is strongly stable, i.e. $c_{\maxi} < n-1$, and let $c'$ be defined as in the statement of Lemma~\ref{lem:add_grain_to_max}. Define $c''$ to be the configuration $c'$ where the values at vertices $\maxi$ and $\nextMax$ are switched, that is:
$$ c''_i = \begin{cases} 
c_{\nextMax} - 1 & \text{if } i = \maxi, \\
c_{\maxi} + 1 & \text{if } i = \nextMax, \\
c_{i} & \text{otherwise}.
\end{cases}$$
Clearly, by symmetry, since $c'$ is SR by Lemma~\ref{lem:add_grain_to_max}, so is $c''$. If we define $k := c_{\maxi} - c_{\nextMax} + 1 \geq 1$, we have:
\beq\label{eq:reduction}
c = \dfrac{k}{k+1} c' + \dfrac{1}{k+1} c''.
\eeq
Thus we have shown that if $c$ is SR and strongly stable, we can write $c$ as a convex sum of two SR states whose maximal number of grains are strictly greater than that of $c$. By iterating this process, we will eventually be able to write $c$ as a convex sum of configurations which are not strongly stable. It is thus sufficient to consider the case where $c$ is not strongly stable, i.e. $c_{\maxi} = n-1$.

Suppose therefore that $c \in \StoRecn$ with $c_{\maxi} = n-1$. To simplify notation, we may assume without loss of generality that $c = \inc{c}$ (since $\DRPoly{n}$ and $\StoRecn$ are invariant under permutation), i.e. $c_1 \leq c_2 \leq c_n = n-1$. We may also assume that $c \notin \DetRecn$. By Corollary~\ref{cor:detrec_comp_burn} we may then define
\beq\label{eq:def_j}
j := \max \{ i \in [n], c_i < i-1 \},
\eeq
and since $c_n = n-1$ we have $j<n$.

By construction, we have that for all $i \leq j$, $c_i \leq c_j < j-1$. Thus the restriction of $c$ to $[j]$, defined by $c_{\big|[j]} := (c_1,\cdots,c_j)$, is a stable configuration on $K_j$. Since Equation~\eqref{eq:SR_complete_char} holds for all $A \subseteq [n]$, it necessarily holds for all $A \subseteq [j]$, and so by Theorem~\ref{thm:SR_complete_char} we have $c_{\big|[j]} \in \mathrm{StoRec}_j$. By the induction hypothesis, we can therefore write:
\beq\label{eq:induction}
c_{\big|[j]} = \sum\limits_{i=1}^k \lambda_i c^{(i)},
\eeq
for $c^{(1)},\ldots,c^{(k)} \in \mathrm{DetRec}_j$, and $\lambda_1,\ldots,\lambda_k \in [0,1]$ with $\sum\limits_{i=1}^k \lambda_i = 1$. Now for $i \in \{1,\ldots,k\}$, define configurations $\tilde{c}^{(i)} \in \Confign$ by
$$ \tilde{c}^{(i)}_{\ell} := \begin{cases}
c^{(i)}_{\ell} & \text{if } \ell \leq j, \\
c_{\ell} & \text{if } \ell > j. \\
\end{cases} $$
By construction, we have $c = \sum\limits_{i=1}^k \lambda_i \tilde{c}^{(i)}$, so it remains to show that $\tilde{c}^{(i)} \in \DetRecn$ for $i \in \{1,\ldots,k\}$. But this follows immediately, using Dhar's burning criterion (Theorem~\ref{thm:DR_general}), from the fact that $c^{(i)} \in \mathrm{DetRec}_j$, and that by construction $c_{\ell} \geq \ell - 1$ for $\ell > j$. Indeed, from this latter fact, we can burn vertices $n,n-1,\ldots,j+1$ in that order, so that we are left with $c^{(i)}$, which is DR, so all remaining vertices are also burned. This concludes the proof of Theorem~\ref{thm:SR_DRPoly}. \qed

\begin{remark}\label{rem:spec_level}
Note that in Equation~\eqref{eq:reduction}, both $c'$ and $c''$ have the same level as $c$. As such, the same proof works when restricting to SR and DR states of a fixed level. That is, the set of SR states with level $k$ (for fixed $k$) is the set of integer lattice points in the polytope of DR states with level $k$. Setting $k=0$ then gives Corollary~\ref{cor:PermPoly_forests}, as explained in Section~\ref{subsec:statement_result}.
\end{remark}


\section{Partially stochastic sandpile models}\label{sec:partial_ssm}

In this section, we introduce a notion of \emph{partially stochastic sandpile model}, in which some of the vertices topple stochastically, and the others deterministically. Let $n \geq 0$ and $k \in [n]_0$. We define the $k$-partial SSM on $K_n$ as the model in which the vertices $\{1, \ldots, k\}$ topple stochastically, i.e. according to Equation~\eqref{eq:part_toppling}, while the vertices $\{k+1, \ldots, n\}$ topple deterministically, i.e. according to Equation~\eqref{eq:full_toppling}. For $k=0$ we get the ASM, while for $k=n$ we get the SSM. We let $\PSR{k}$ denote the set of recurrent states for the $k$-partial SSM, and refer to its elements as $k$-SR states.

Since partial (stochastic) topplings always have a positive probability of being full (deterministic), elementary Markov chain theory gives us the following.

\begin{proposition}\label{pro:inc_part_storec}
The sequence $\left( \PSR{k} \right)_{0 \leq k \leq n}$ is increasing for inclusion, i.e.
\beq\label{eq:inc_part_storec}
\DetRecn = \PSR{0} \subseteq \PSR{1} \subseteq \cdots \subseteq \PSR{n} = \StoRecn.
\eeq
\end{proposition}

We know that for $n \geq 3$, we have $\DetRecn \subsetneq \StoRecn$, so at least one of these inclusions must be strict. In fact, we can show the following.

\begin{proposition}\label{pro:part_sto_non_det}
For $n \geq 3$, we have $\DetRecn \subsetneq \PSR{1}$. In words, having just one vertex topple stochastically yields a recurrent state that is not DR.
\end{proposition}

\begin{proof}
Fix $n \geq 3$. We define a stable configuration $c = (c_1, \ldots, c_n) \in \Stablen$ by $c_1 = c_2 = c_3 = 1$ and $c_i = n-1$ for all $i > 3$ (maximal stable number of grains). First, note that $c$ is not DR, since the induced cycle $\{1, 2, 3\}$ is a forbidden subconfiguration (for Condition~(iii) of Theorem~\ref{thm:DR_general}). It is therefore sufficient to show that $c \in \PSR{1}$, i.e. that $c$ is recurrent for the partially stochastic sandpile model in which vertex $1$ topples stochastically, and all other vertices deterministically.

First, we define a configuration $c'$ by $c' := c - \alpha_1 + \alpha_2$, i..e. $c'_1 = 0$, $c'_2 = 2$, $c'_3 = 1$, and $c'_i = n-1$ for all $i > 3$. We claim that $c'$ is DR by applying Dhar's burning algorithm. After burning the sink, all vertices $i > 3$ are unstable, and can be burned. This leaves us with the cycle $\{1, 2, 3\}$ again, but this time it can be burned in order $2, 3, 1$.

Since $c'$ is DR, there exists a configuration $d = (b_1, \ldots, b_n)$ such that $\DetStab(\maxc + d) = c$ by Condition~(i) of Theorem~\ref{thm:DR_general}. Moreover, since $c'_1 = 0$, the last vertex to topple in the deterministic stabilisation of $\maxc + d$ must be vertex $1$. Indeed, toppling any other vertex sends a grain to $1$, so if the last toppling in the stabilisation is not at $1$ we would have $c'_1 > 0$. 

We now consider the same stabilisation sequence up to the last toppling of $1$ excluded, and change that toppling of $1$ so that it sends a grain to each neighbour except vertex $2$ (this is now a stochastic toppling which has positive probability of occurring). This modifies the resulting configuration by moving one grain from $2$ to $1$ in $c'$, yielding exactly the configuration $c$. We have thus shown that $c$ can be obtained from $\maxc$ through a finite sequence of grain additions and topplings for the $1$-partial SSM, meaning that $c$ is recurrent for this model, as desired.
\end{proof}

We can then look into the right-hand end of the sequence in~\eqref{eq:inc_part_storec}, i.e.\ for which $k$ we have $\PSR{k} = \StoRecn$. First, we show that if at least three vertices topple deterministically, then not all SR configurations are reached.

\begin{proposition}\label{pro:part_sto_non_sr}
Let $n \geq 3$, and define a configuration $c = (c_1, \ldots, c_n) \in \Confign$ by 
$$ c_k = 
\begin{cases} n-1 & \text{if } k \leq n-3,\\
1 & \text{otherwise.} 
\end{cases} $$
Then $c \in \StoRecn \setminus \PSR{n-3}$.
\end{proposition}

\begin{proof}
Let $c$ be as in the statement of the proposition. We first show that $c$ is SR. For this, consider the orientation $\OO$ of $K_n$ such that $n-2 \orighta n-1 \orighta n \orighta n-2$, $j \orighta k$ if $j \geq n-2$ and $k \leq n-3$, and the edges $\{k, k'\}$ for $k, k' \leq n-3$ are oriented in any arbitrary manner. In words, we construct a directed cycle on the last three vertices, orient all other edges incident to one of those vertices as outgoing (away from the cycle), and choose an arbitrary direction for the remaining edges. By construction we have $\In{j} = 1$ for any $j \geq n-2$, and since $c_k = n-1 \geq \In{k}$ for any $k \leq n-3$, the orientation $\OO$ is compatible with $c$, and $c$ is SR by Theorem~\ref{thm:SR_general}.

We now show that $c$ is not $(n-3)$-SR. Seeking contradiction, assume that $c$ is $(n-3)$-SR. This implies that $c$ can be reached from $\maxc$ through a sequence of grain additions and topplings, where the topplings are stochastic for the first $(n-3)$ vertices, and deterministic for the last $3$. But since each of the last $3$ vertices has one grain of sand, they each must topple at least once in this sequence. Consider the final topplings of each of these, and without loss of generality assume that this occurs in order $n-2, n-1, n$. Since the topplings of $(n-1)$ and $n$ both send a grain to $(n-2)$, and since $(n-2)$ doesn't subsequently topple (we considered the final topplings in the sequence leading to $c$), then $(n-2)$ must have at least two grains in the final stable configuration reached, i.e. $c_{n-2} \geq 2$. This gives the desired contradiction ($c_{n-2} = 1$ by definition), and the proposition is proved.
\end{proof}

\begin{remark}\label{rem:ess_cycle}
The proof of Proposition~\ref{pro:part_sto_non_sr} indicates that one thing stopping a configuration from being partially SR is a $(1, 1, 1)$ cycle on vertices which topple fully. In fact, it is a little bit more complex than that. For example, for $n=4$, the configurations which are SR, but not $1$-SR, are $(3, 1, 1, 1)$ (the counter-example from Proposition~\ref{pro:part_sto_non_sr}), and $(0, 2, 2, 2)$. The orientation compatible with the first is as in the proof above, while the orientation compatible with the second is $2 \orighta 3 \orighta 4 \orighta 2$, and $1 \orighta j$ for $j \in \{2, 3, 4\}$. Note that this configuration is minimal recurrent, so this orientation is unique up to a flipping of the $(2, 3, 4)$-cycle.
\end{remark}

The above suggests that what stops a configuration from being partially SR is in fact precisely a ``forced'' cycle on vertices which topple fully. We make this observation precise in the following technical lemma. 

\begin{lemma}\label{lem:forced_cycle}
Let $n \geq 3$ and $\ell \geq 1$. We define $K_n^{(\ell)}$ to be the complete graph on $[n]_0$ where for every $k \in [n]$, the edge $\{k, 0\}$ connecting $k$ to the sink has multiplicity $\ell$. Note that if $\ell = 1$ this is just the usual complete graph $K_n^0$. 
Let $c \in \StoRec{K_n^{(\ell)}}$ be SR on this graph. Let $k_1, k_2, k_3 \in [n]$ be three \emph{distinct} vertices in $[n]$. 
Suppose that there exists an orientation $\OO$ of $K_n$, compatible with $c$, such that $k_1 \orighta k_2$, $k_2 \orighta k_3$, and $k_1 \orighta k_3$. Then there exists a sequence of grain additions and topplings leading from $\maxc$ to $c$ such that the vertices $k_1, k_2, k_3$ always topple deterministically. Moreover, the final topplings of these three vertices are in the order $k_3, k_2, k_1$, possibly with $k_3$ never toppling.
\end{lemma}

\begin{proof}
First, note that it is sufficient to show the lemma for minimal SR states $c$. Indeed, any recurrent configuration can be reached from a certain minimal recurrent configuration through grain additions, and putting extra grain additions at the end of the additions and topplings sequence doesn't change the conclusion of the lemma. Thus we may assume that an orientation $\OO$ compatible with $c$ satisfies $c_i = \In{i}$ at all vertices $i \in [n]$.

We proceed by induction on $n \geq 3$. First, let us show that the statement holds for $n=3$. Fix $\ell \geq 1$ and let $c \in \StoRec{K_3^{(\ell)}}$ be minimal recurrent. Without loss of generality, we choose $k_1 = 1, k_2 = 2, k_3=3$ for the statement. Then, by definition of $\OO$, we have $c_1 = 0, c_2=1, c_3=2$. It is straightforward to check that $c$ is DR (e.g.\ by applying Dhar's burning algorithm). 
Therefore, there exists a sequence of grain additions and full topplings leading from $\maxc$ to $c$ in $K_3^{(\ell)}$. Moreover, since $c_1 = 0$, vertex $1$ must be the last vertex to topple in this sequence, and since $c_2 = 1$, $2$ must be the penultimate vertex to topple (it receives one grain from the subsequent toppling of $1$, and can't receive any other grains). This proves the statement in the case $n=3$.

Now fix some $n \geq 4$, $\ell \geq 1$, and assume that the statement holds for $(n-1)$ and all $\ell' \geq 1$. Let $c \in \StoRec{K_n^{(\ell)}}$ be minimal recurrent. Let $k_1, k_2, k_3$ and the orientation $\OO$ of $K_n$ be as in the statement of the lemma, with $c_i = \In{i}$ for all $i \in [n]$. Choose some $m \notin \{k_1, k_2, k_3\}$, and to simplify notation, assume that we can take $m=n$. Let $S \subseteq [n-1]$ be the set of vertices $i$ such that $n \orighta i$, and $T := [n-1] \setminus S$. Note that by definition $c_n = \In{n} = \vert T \vert = n-1-\vert S \vert$. We define a configuration $c' = (c'_1, \ldots, c'_{n-1}) \in \Config{K_{n-1}^{(\ell + 1)}}$ by 
\beq\label{eq:tech_lemma_def_c'}
c'_i := \begin{cases}
c_i - 1 & \text{if } i \in S, \\
c_i & \text{otherwise}.\\
\end{cases} 
\eeq

Let $\OO'$ be the orientation on $K_{n-1}$ obtained by simply removing all edges incident to $n$ in $\OO$. By definition of $S$, we have:
\beq\label{eq:tech_lemma_comp}
\mathrm{in}_i^{\OO'} = \begin{cases}
\In{i} -1 & \text{if } i \in S, \\
\In{i} & \text{otherwise}.\\
\end{cases} 
\eeq
But since $\In{i} = c_i$ for all $i \in [n]$ by assumption, Equations~\eqref{eq:tech_lemma_def_c'} and \eqref{eq:tech_lemma_comp} immediately imply that $ \mathrm{in}_i^{\OO'} = c'_i$ for all $i \in [n-1]$, and so $c'$ is minimal recurrent on $K_{n-1}^{(\ell + 1)}$ by Theorem~\ref{thm:char_minrec_general_orientation}. 

Now note that $k_1, k_2, k_3$ and the orientation $\OO'$ still satisfy the conditions of Lemma~\ref{lem:forced_cycle} applied to $c'$ and $K_{n-1}^{(\ell + 1)}$. We may therefore apply the induction hypothesis, meaning that $c'$ can be reached from $\maxc$ (on $K_{n-1}^{(\ell + 1)}$) through a series of grain additions and topplings, in which $k_1$, $k_2$ and $k_3$ always topple deterministically, and their final topplings are in the order $k_3, k_2, k_1$. We will write $\mathrm{SEQ}'$ to denote this sequence.

Now note that, by definition of $K_n^{(\ell)}$, for any $i \in [n-1]$, we have $\dgr{i}^{K_n^{(\ell)}} = \dgr{i}^{K_{n-1}^{(\ell + 1)}}$ (the extra edge to the sink in the latter cancelling out the edge to $n$ in the former). Therefore, we can ``copy'' the above sequence $\mathrm{SEQ}'$ of grain additions and topplings into a sequence $\mathrm{SEQ}$ from $\maxc$ on $K_n^{(\ell)}$, as follows. Grain additions are unchanged, as are any topplings that send at most $\ell$ grains to the sink. Then, if $\ell + 1$ vertices are sent to the sink in a single toppling in the sequence $\mathrm{SEQ}'$, in the ``copy'' sequence $\mathrm{SEQ}$ we send $\ell$ grains to the sink, and one grain to vertex $n$. With this construction, $\mathrm{SEQ}$ is a legal sequence of grain additions and topplings on $K_n^{(\ell)}$, and by construction we reach a configuration $c'' \in \Config{K_n^{(\ell)}}$ such that $c''_i = c'_i$ for all $i \in [n-1]$ (grain movements between vertices of $[n-1]$ are unchanged). Note also that any deterministic topplings in $\mathrm{SEQ}'$ translate to deterministic topplings in $\mathrm{SEQ}$. Moreover, the order of topplings is unchanged, so that in $\mathrm{SEQ}$ the final topplings of $k_1$, $k_2$ and $k_3$ are still in the order $k_3, k_2, k_1$.

Depending on the value of $c''_n$, we either add grains to vertex $n$, or repeatedly send a single grain from $n$ to the sink (this is a stochastic toppling that has positive probability), until $c''_n = n - 1 + \ell = \dgr{n}^{K_n^{(\ell)}}$.
We then add one final toppling to our sequence by toppling vertex $n$ stochastically, sending $\ell$ grains to the sink, and one grain to each vertex in $S$. This means that the number of grains remaining at $n$ after this toppling is exactly $(n - 1 + \ell) - (\ell + \vert S \vert) = \vert T \vert = c_n$. Moreover, by construction, if $i \in S$, then $i$ now has $1 + c''_i = 1 + c'_i = c_i$ grains, and if $i \in T$, then $i$ now has $c''_i = c'_i = c_i$ grains. In other words, after this final toppling we have exactly reached the configuration $c$, and the series of grain additions and topplings applied satisfies the conclusions of Lemma~\ref{lem:forced_cycle}. This completes the proof.
\end{proof}

In fact, we only need to apply this in the case $\ell = 1$, but as we saw the general case was necessary for the proof since the reduction from $n$ to $(n-1)$ forced the increase from $\ell$ to $(\ell + 1)$. We are now equipped to state our main result of this section, namely that $(n-2)$ stochastically-toppling vertices are sufficient to reach all SR configurations.

\begin{theorem}\label{thm:par_sto_rec_all_sto_rec}
Let $n \geq 3$. We have
\beq\label{eq:par_sto_rec_all_sto_rec}
\PSR{n-2} = \StoRecn.
\eeq
\end{theorem}

\begin{proof}
We know from Proposition~\ref{pro:inc_part_storec} that $\PSR{n-2} \subseteq \StoRecn$. For the converse, let $c \in \StoRecn$ be SR. It is once again sufficient to consider the case where $c$ is minimal recurrent. Therefore, by Theorem~\ref{thm:char_minrec_general_orientation}, there exists an orientation $\OO$ of $K_n$ satisfying $\In{k} = c_k$ for all vertices $k$. Suppose that there exists such an orientation where the last three vertices $n-2$, $n-1$, and $n$ do not form a directed cycle. Then from Lemma~\ref{lem:forced_cycle}, we can find a series of grain additions and topplings leading from $\maxc$ to $c$ such that the last three vertices always topple deterministically, implying that $c \in \PSR{n-3} \subseteq \PSR{n-2}$.

We may therefore assume that any orientation compatible with $c$ has a directed cycle on the last three vertices. Without loss of generality, assume that $n-2 \orighta n-1 \orighta n \orighta n-2$. Flipping the edge $n \orighta n-2$ of $\OO$ yields an orientation $\OO'$, and the configuration compatible with $\OO'$ is $c' = c + \alpha_n - \alpha_{n-2}$. But the orientation $\OO'$ now satisfies the conditions of Lemma~\ref{lem:forced_cycle}. Therefore, there exists a series of grain additions and topplings leading from $\maxc$ to $c'$ such that the vertices $(n-2)$, $(n-1)$ and $n$ always topple deterministically. Moreover, the last three topplings of those vertices are $n, n-1, n-2$ in that order. 

If we change the final toppling of $(n-2)$ so that it now sends grains to all vertices $k < n$, but doesn't send a grain to $n$, this changes the final configuration from $c'$ to $c$. Note that any toppling occurring after that of $(n-2)$ can still occur, since the only change is $n$ no longer receiving a grain, and $n$'s last toppling is before that of $(n-2)$. As such, we have constructed a sequence of grain additions and topplings leading from $\maxc$ to $c$ in which vertices $(n-1)$ and $n$ always topple deterministically (since this was the case in the sequence leading to $c'$, and we left those topplings unchanged). This means exactly that $c \in \PSR{n-2}$, and the theorem is proved.
\end{proof}

\begin{remark}\label{rem:sto_strength}
The proof of Theorem~\ref{thm:par_sto_rec_all_sto_rec} seems to suggest that SR configurations don't require many stochastic topplings to be reached. Indeed, in the proof we only required one stochastic toppling on the ``forced'' cycle, and in fact this toppling in question was ``almost deterministic'', since it kept only one grain at the toppling vertex. It would be interesting to formalise this result by stating how many non-deterministic topplings are needed in general to reach a SR configuration from $\maxc$, and how close these topplings are to being deterministic.
\end{remark}


\section{Conclusion and future perspectives}\label{sec:conc}

In this paper, we have considered a stochastic variant of the ASM, called SSM, in which toppling vertices flip a biased coin for each neighbour and decide with probability $p$ to send a grain to that neighbour, and with probability $(1-p)$ to keep the grain (here $p \in (0, 1)$ is fixed). We have studied the SSM on complete graphs from a combinatorial point of view, focusing on characterisations of the model's recurrent states (called SR states). To our knowledge, this is the first such combinatorial study of a stochastic sandpile model.

We exhibited a stochastic version of the burning algorithm for complete graphs that establishes in log-linear time whether a given state is SR or not (Theorem~\ref{thm:sto_burning}). We then showed that the set of SR states is the set of integer lattice points of the convex polytope of DR states (Theorem~\ref{thm:SR_DRPoly}). In other words, any SR state can be written as a convex sum of DR states. Finally, we considered a family of partially stochastic sandpile models on complete graphs, in which some vertices always topple deterministically according to the ASM rules, while others topple stochastically according to the SSM rules. We showed that in general the set of recurrent states for these models are distinct from those of both the ASM and SSM (Propositions~\ref{pro:part_sto_non_det} and \ref{pro:part_sto_non_sr}), and that if at most two vertices topple deterministically, we get all SR states (Theorem~\ref{thm:par_sto_rec_all_sto_rec}).

There are a number of open questions and directions to be considered by future research. The first is that of more general stochastic burning algorithms, i.e.\ algorithms that successively burn (remove) vertices until either all vertices are burned (in which case the starting configuration is SR),  or we reach a forbidden subconfiguration where no more vertices can be burned (in which case the starting configuration is not SR). Remark~\ref{rem:sto_burning_only_complete} explains why two of the more natural burning criteria don't work in the general case.

Another future research direction that seems natural is to study the SSM on other graph families. Note that the author has in fact considered the SSM on wheel graphs in~\cite{SelWheel}, but in this case the difference with the ASM is minimal (there is only one additional recurrent state). It would be interesting to consider for instance the case of complete bipartite graphs (with or without a dominating sink), complete split graphs, etc., as has been done for the ASM.

Another open question is the existence of an explicit enumeration formula for the set of SR states $\StoRecn$. Such a formula is given in terms of permutohedron lattice points enumeration in~\cite[Theorem~5.1]{AW}, but it is rather complicated. One possible approach to this would be to use the Tutte-like deletion-contraction relationship given for the so-called \emph{lacking polynomial} in \cite[Theorem~3.9]{CMS} (the lacking polynomial is equivalent to the level polynomial for the SSM, up to a change of variable). One would hope that this could lead to a reasonably simple recursive formula, but so far this approach has not been successful. If no such explicit enumeration formula can be found, can asymptotic estimates be given?

Finally, it would be interesting to study the partially stochastic sandpile models further. Table~\ref{table:par_ssm} below gives the first few values of $\vert \PSR{1} \vert$ for $n = 1, \ldots, 6$. This sequence is not known in the OEIS~\cite{OEIS}. We observe that in general the numbers of $1$-SR states appear a fair amount closer to the numbers of SR states than they are to those of DR states, albeit our data is quite limited (these numbers are very slow to compute). This is somehow consistent with Remark~\ref{rem:sto_strength} that SR states in general don't require many stochastic topplings to be reached, so a single vertex toppling stochastically should be enough to reach most of them. It would be interesting to quantify this observation in some way, perhaps in terms of asymptotic behaviour of these numbers.

\begin{table}[ht]
\centering
\begin{tabular}{c|c|c|c}
$n$ & $\left\vert \DetRecn \right\vert$ & $ \left\vert \PSR{1} \right\vert $ & $\left\vert\StoRecn \right\vert $ \\ [8pt]
\hline
1 & 1 & 1 & 1 \\
2 & 3 & 3 & 3 \\
3 & 16 & 17 & 17 \\
4 & 125 & 142 & 144 \\
5 & 1296 & 1563 & 1623 \\
6 & 16807 & 21326 & 22804\\
\end{tabular}
\caption{The numbers of DR, $1$-SR, and SR states on $K_n$, for $n$ ranging from $1$ to $6$.\label{table:par_ssm}}
\end{table}

\subsection*{Acknowledgements}

The research leading to these results is partially supported by the National Natural Science Foundation of China, grant number 12101505, and by the Research Development Fund of Xi'an Jiaotong-Liverpool University, grant number RDF-22-01-089.


\end{document}